\documentclass[a4paper,12pt]{scrartcl}

\usepackage[utf8]{inputenc}
\usepackage[english]{babel}
\usepackage{amssymb}
\usepackage{amsmath}
\usepackage{latexsym}
\usepackage{amsthm}
\usepackage{eucal}
\usepackage{graphicx}
\usepackage{bbm}
\usepackage{verbatim}
\usepackage{epigraph}
\usepackage{mathtools}
\usepackage{authblk}
\usepackage[pdflatex]{crop}
\usepackage[bookmarks]{hyperref}
\usepackage{tikz}
\usepackage{tikz-cd}
\usepackage{wasysym}
\usepackage{adjustbox}
\usepackage[noabbrev]{cleveref}
\allowdisplaybreaks
\usetikzlibrary{shapes.geometric,fit}

\setkomafont{disposition}{\normalfont\bfseries}
\newcommand{\auth}[0]{Tobias Fritz and Paolo Perrone}
\newcommand{\tit}[0]{Monads, partial evaluations, and rewriting}
\newcommand{\kw}[0]{Monads, bar construction, rewriting, higher dimensional rewriting, simplicial sets.}

\hypersetup{
pdfauthor={\auth},%
pdftitle={\tit},%
colorlinks, linktocpage=true, pdfstartpage=1, pdfstartview=FitV,%
breaklinks=true, pdfpagemode=UseNone, pageanchor=true, pdfpagemode=UseOutlines,%
plainpages=false, bookmarksnumbered, bookmarksopen=true, bookmarksopenlevel=1,%
hypertexnames=true, pdfhighlight=/O,%
urlcolor=black, linkcolor=black, citecolor=black, 
}
\numberwithin{equation}{section}

\theoremstyle{plain}
\newtheorem{thm}{Theorem}[section]

\newtheorem{prop}[thm]{Proposition}
\newtheorem{cor}[thm]{Corollary}

\newtheorem{deph}[thm]{Definition}
\newtheorem{prob}[thm]{Problem}
\theoremstyle{deph}

\newtheorem{eg}[thm]{Example}

\newcommand{\N}{\mathbb{N}}

\newcommand{\R}{\mathbb{R}}

\newcommand{\cat}[1]{{\mathsf{#1}}} 
\newcommand{\ar}[2][]{\arrow{#2}{#1}}

\newcommand{\id}{\mathrm{id}} 

\mathchardef\hy="2D

\usepackage{todonotes}

\let\originalleft\left
\let\originalright\right
\renewcommand{\left}{\mathopen{}\mathclose\bgroup\originalleft}
\renewcommand{\right}{\aftergroup\egroup\originalright}

\tikzset{%
    bullet/.style={
       fill=black,
       circle,
       minimum width=1pt,
       inner sep=1pt
     },
     relation/.style={
       -,
       thick,
       shorten <=2pt,
       shorten >=2pt
     },
     function/.style={
       ->,
       thick,
       shorten <=2pt,
       shorten >=2pt
     },
     every fit/.style={
       ellipse,
       draw,
       inner sep=0pt
     }
}

\title{\vspace{-1cm}\tit}
\author[1]{Tobias Fritz\footnote{Correspondence: tfritz [at] perimeterinstitute.ca}}
\author[2]{Paolo Perrone\footnote{Correspondence: pperrone [at] mit.edu}}

\affil[1]{\small Perimeter Institute for Theoretical Physics, Waterloo, ON, Canada}
\affil[2]{\small Massachusetts Institute of Technology, Cambridge, MA, U.S.A.}
\date{MFPS 2020}

\begin{document}

\maketitle


\begin{abstract}
\addcontentsline{toc}{section}{Abstract}

Monads can be interpreted as encoding formal expressions, or formal operations in the sense of universal algebra. We give a construction which formalizes the idea of ``evaluating an expression partially'': for example, ``2+3'' can be obtained as a partial evaluation of ``2+2+1''. 
This construction can be given for any monad, and it is linked to the famous \emph{bar construction}~\cite[VII.6]{maclane}, of which it gives an operational interpretation: the bar construction is a simplicial set, and its 1-cells are partial evaluations.

We study the properties of partial evaluations for general monads. We prove that whenever the monad is weakly cartesian, partial evaluations can be composed via the usual Kan filler property of simplicial sets, of which we give an interpretation in terms of substitution of terms. 

For the case of probability monads, partial evaluations correspond to what probabilists call \emph{conditional expectation} of random variables, and partial evaluation relation is known as \emph{second-order stochastic dominance}.

In terms of rewritings, partial evaluations give an abstract reduction system which is reflexive, confluent, and transitive whenever the monad is weakly cartesian. This manuscript is part of a work in progress on a general rewriting interpretation of the bar construction.


\end{abstract}

\tableofcontents

\section{Background: monads and formal expressions}\label{secformop}

An interpretation of the theory of monads, in terms of universal algebra~\cite{hyland-power}, is that \emph{a monad is like a consistent choice of spaces of formal expressions in a signature}. 
This interpretation is most accurate for monads on the category of sets, but the categorical constructions work in general.

In more detail, a monad consists of the following data. First of all, we have a functor $T:\cat{C}\to\cat{C}$, which consists of the following assignments:
\begin{enumerate}
 \item To each space $X$, we assign a new space $TX$, which we think of as containing formal expressions of elements of $X$ in a certain signature, modulo the equations specified by the theory. For example, the elements of $TX$ may exactly be the formal sums of elements of $X$.

 \item Given two spaces $X$ and $Y$ and a function $f:X\to Y$, we get a function $Tf:TX\to TY$, which we think of as elementwise substitution. This assignment should preserve identity and composition.
 In the case of formal sums, given a function $f:X\to Y$, we automatically get a function from formal sums of elements of $X$ to formal sums of elements of $Y$ by just ``extending linearly''. For example:
\begin{equation}
 a + b + 2c \quad\longmapsto\quad f(a) + f(b) + 2f(c) .
\end{equation}
\end{enumerate}

In the case of formal sums, any element $x$ can be considered a (trivial) formal sum. For general monads, this is encoded in the unit natural transformation $\eta:\id_{\cat{C}}\Rightarrow T$. Moreover, formal sums of formal sums can be reduced to just formal sums, such as
\begin{equation*}
 ( a + b + c ) + ( a + b + d ) 
\end{equation*}
can be reduced to 
\begin{equation*}
 2 a + 2 b + c + d .
\end{equation*}
This in general is encoded in the monad multiplication transformation $\mu:TT\Rightarrow T$. The unit and multiplication transformations are then required to satisfy the monad equations
\begin{equation}
	\label{monad_laws}
	\begin{tikzcd}
		TTTX \ar{r}{T\mu} \ar{d}{\mu}	& TTX \ar{d}{\mu}	 & TX \ar{d}[swap]{\eta} \ar[equal]{dr}	&	& TX \ar{d}[swap]{T\eta} \ar[equal]{dr}	\\
		TTX \ar{r}{\mu}		& TX				& TTX \ar{r}{\mu}		& TX		& TTX \ar{r}{\mu}	& TX
	\end{tikzcd}
\end{equation}
for all $X$ in $\cat{C}$.

In general, the category $\cat{C}$ need not be concrete. Therefore, instead of elements, we look at generalized elements: in general, what we can interpret as a ``formal expression'' should be a morphism $S\to TX$ for some object $S$ of $\cat{C}$. Let's have the following convenient definition:
\begin{deph}
 Let $(T,\eta,\mu)$ be a monad on a category $C$, and $X$ be an object. A \emph{generalized formal expression on $X$} is a morphism $p:S\to TX$, where $S$ is an object of $C$. 
\end{deph}

If $T$ lives on $\cat{Set}$, then for $S$ we can take a singleton set $1$, recovering the usual formal expressions. The same can be done for many concrete categories. Nevertheless, we will drop the word ``generalized'' and just speak of morphisms $S \to TX$ for general $\cat{C}$ as \emph{formal expressions}.

Often, formal expressions can be \emph{evaluated} to a \emph{result}. For example, the expression $3+2$ can be evaluated to $5$. 
An \emph{algebra} of a monad is an object in which (generalized) formal expressions can be evaluated to actual (generalized) elements. So for the formal sum monad on $\cat{Set}$, the category of algebras is precisely the category of commutative monoids, since in commutative monoids formal sums can be evaluated to actual elements in a way which respects usual rewriting rules for working with sums, namely associativity and commutativity.

Formally, an algebra of the monad $T$ consists of an object $A$ together together with an evaluation map $e:TA\to A$, suitably compatible with $\eta$ and $\mu$ in the sense of making the diagrams
\[
	\begin{tikzcd}
		TTA \ar{r}{Te} \ar{d}{\mu}	& TA \ar{d}{e}		& & A \ar{r}{\eta}\ar{dr}[swap]{\id} & TA \ar{d}{e}	\\
		TA \ar{r}{e}			& A			& & 	& A
	\end{tikzcd}
\]
commute. By \eqref{monad_laws}, every object $TX$ is an algebra with respect to $\mu : TTX \to TX$, the \emph{free algebra} on $X$.

\begin{deph}
 Let $(A,e)$ be a $T$-algebra. Given a (generalized) formal expression $p:S\to TA$, we call its \emph{result} the (generalized) element of $A$ given by $e\circ p: S\to A$.
\end{deph}

\section{Partial evaluations and partial decompositions}\label{secpev}

Consider the sums 
\begin{equation}\label{pdecsum}
3+4+5
\end{equation}
and 
\begin{equation}\label{pevsum}
 7+5 .
\end{equation}
Not only do they have the same result, but in addition, we can say that the sum~\eqref{pevsum} can be obtained from~\eqref{pdecsum} by \emph{partially evaluating the expression}. Just as well, we would like to say that the sum~\eqref{pdecsum} can be obtained from~\eqref{pevsum} by \emph{partially decomposing the terms in the expression}.

Let's try to make this precise. The idea is that there is a \emph{formal sum of formal sums}, i.e.~a formal sum with one level of brackets, such that removing the brackets yields the term on the left, and such that performing the operations in the brackets (and then removing the brackets) yields the term on the right. That is:
\begin{equation*}
 \begin{tikzcd}[row sep=large]
  & (3+4) + (5) \ar{dl}[swap]{\mbox{remove brackets}} \ar{dr}{\mbox{evaluate brackets}} \\
  3+4+5 && 7+5
 \end{tikzcd}
\end{equation*}
As we have seen in Section~\ref{secformop}, the ``formal sums of formal sums'' live in $TTA$. The map which can be seen as ``removing the brackets'' is the multiplication map $\mu:TTA\to TA$, and the map that evaluates the expressions within the brackets is the image of the evaluation map $e$ under the functor $T$, i.e.~$Te:TTA\to TA$. We can then give a general definition of partial evaluations for all monads. 
\begin{deph}\label{defpev}
 Let $(T,\eta,\mu)$ be a monad on a category $\cat{C}$, let $(A,e)$ be a $T$-algebra, and consider the formal expressions $p,q:S\to TA$. A \emph{partial evaluation} of $p$ into $q$, or a \emph{partial decomposition} of $q$ into $p$ is a map $k:S\to TTA$, or ``nested formal expression'', which makes the following diagram commute:
 $$
 \begin{tikzcd}
  & S \ar{dl}[swap]{p} \ar{d}{k} \ar{dr}{q} \\
  TA & TTA \ar{l}{\mu} \ar{r}[swap]{Te} & TA
 \end{tikzcd}
 $$
\end{deph}

\subsection{Basic properties}

From the definition and the triangle identities we have immediately the following result, which is a sort of consistency check: any expression has two trivial partial evaluations, namely to itself and to its result (viewed as a formal expression). 

\begin{prop}\label{trivialpev}
 Let $(A,e)$ be a $T$-algebra as above, and $p:S\to TA$. Then:
 \begin{enumerate}
  \item $p$ admits a partial evaluation to itself;
  \item $p$ admits a partial evaluation to $\eta\circ e \circ p$, which we call its \emph{total evaluation}.
 \end{enumerate}
\end{prop}

\begin{proof}
$ $
 \begin{enumerate}
	 \item Consider $(T\eta)\circ p:S\to TTA$. Then $\mu\circ (T\eta)\circ p = p$ by the right unitality~\eqref{monad_laws} of the monad, and $(Te)\circ(T\eta)\circ p = T(e\circ\eta)\circ p = p$ by functoriality of $T$ together with the unit condition of the algebra. Therefore, $(T\eta)\circ p$ gives a partial evaluation of $p$ into itself.
  \item Consider $\eta\circ p:S\to TTA$. We have a diagram
  \begin{equation*}
   \begin{tikzcd}
    TA \ar{d}{e} \ar{r}{\eta} & TTA \ar{d}{Te} \\
    A \ar{r}{\eta} & TA
   \end{tikzcd}
  \end{equation*}
  which commutes by naturality of $\eta$. Now $\mu\circ\eta\circ p=p$ by the left unitality of the monad, and $(Te)\circ\eta\circ p=\eta\circ e\circ p$ by the commutativity of the diagram above. Therefore $\eta\circ p$ gives a partial evaluation of $p$ into $\eta\circ e \circ p$. 
 \end{enumerate}
\end{proof}

\begin{eg}\label{egsum}
	Consider the free commutative monoid monad (or multiset or bag monad) on $\cat{Set}$, take the algebra $(A,e)$ to be the set of natural numbers $\N$ under addition, and let $S=1$, the terminal set. Then the formal expressions are sums of natural numbers, considered as lists up to permutation. The expression $3+4+5$ admits a partial evaluation into itself, given by
 $$
 
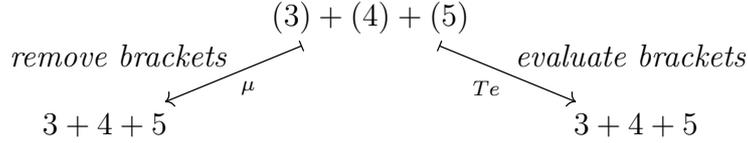
\begin{tikzcd}
  & (3) + (4) + (5) \ar[mapsto]{dl}{\mu}[swap]{\mbox{remove brackets}} \ar[mapsto]{dr}{\mbox{evaluate brackets}}[swap]{Te} \\
  3+4+5 && 3+4+5
 \end{tikzcd}
 $$ 
 and a partial evaluation into its result $12$, or \emph{total evaluation}, given by 
 $$
 
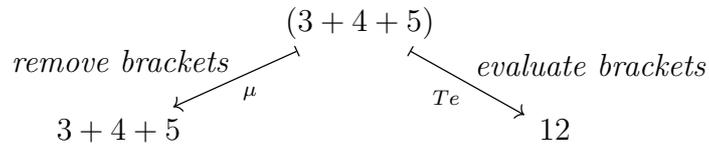
\begin{tikzcd}
  & (3+4+5) \ar[mapsto]{dl}{\mu}[swap]{\mbox{remove brackets}} \ar[mapsto]{dr}{\mbox{evaluate brackets}}[swap]{Te} \\
  3+4+5 && 12
 \end{tikzcd}
 $$
\end{eg}

Here is another consistency check: if $p$ admits a partial evaluation into $q$, then $p$ and $q$ necessarily must have the same result.  

\begin{prop}[Law of total evaluation]\label{totalev}
 Consider the formal expressions $p,q:S\to TA$, and suppose that there exists a partial evaluation from $p$ into $q$. Then $p$ and $q$ have necessarily the same result, i.e.~$e\circ q=e\circ p$. 
\end{prop}

\begin{proof}
 The multiplication square of the $T$-algebra $(A,e)$ is a commutative diagram
 \begin{equation*}
  \begin{tikzcd}
   TTA \ar{d}{\mu} \ar{r}{Te} & TA \ar{d}{e} \\
   TA \ar{r}{e} & A
  \end{tikzcd}
 \end{equation*}
 Now suppose that $k:S\to TTA$ gives a partial evaluation of $p$ into $q$, i.e.~$\mu\circ k = p$, and $(Te)\circ k=q$. Then, since the square above commutes,
 \begin{align*}
  e\circ q = e\circ (Te)\circ k = e\circ \mu \circ k = e\circ p ,
 \end{align*}
as was to be shown.
\end{proof}

\begin{eg}
  Following Example~\ref{egsum}, there is a partial expression of $3+4+5$ into $7+5$, witnessed by $(3+4) + (5)$:
 $$
 \begin{tikzcd}
  & (3+4) + (5) \ar[mapsto]{dl}{\mu}[swap]{\mbox{remove brackets}} \ar[mapsto]{dr}{\mbox{evaluate brackets}}[swap]{Te} \\
  3+4+5 && 7+5 
 \end{tikzcd}
 $$
 The two formal expressions have then necessarily the same result, in this case, $12$.
\end{eg}

\subsection{Composing partial evaluations}

There is another appealing property to expect from partial evaluations, namely that if we can partially evaluate $p$ to $q$ and $q$ to $r$, then we expect that we should be able to partially evaluate $p$ to $r$.  

\begin{eg}\label{egcompose}
 Consider again sums of natural numbers. Then the formal sum $1+1+1+1$ can be partially evaluated to $2+2$, and $2+2$ can be partially (and totally) evaluated to $4$. This is witnessed by the following elements:
 
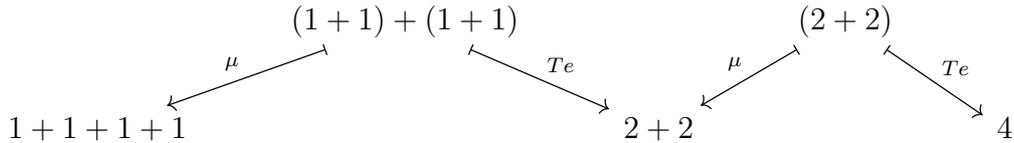
\begin{equation*}
  \begin{tikzcd}
  & (1+1) + (1+1) \ar[mapsto]{dl}[swap]{\mu} \ar[mapsto]{dr}{Te} && (2+2) \ar[mapsto]{dl}[swap]{\mu} \ar[mapsto]{dr}{Te} \\
   1+1+1+1 && 2+2 && 4
  \end{tikzcd}
 \end{equation*}
In order to construct a partial evaluation from the left to the right, we need an expression such that removing its brackets yields $1+1+1+1$, and evaluating the brackets yields $4$. The idea is that $4$ is obtained by $2+2$, but each of the $2$ in the expression is itself obtained by $1+1$. Therefore we can \emph{substitute} each term $2$ by the term of which it is a partial evaluation, in other words, we form the expression $((1+1)+(1+1))$ which sits in the diagram:
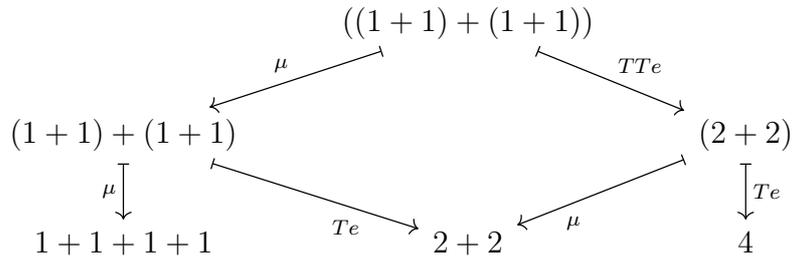
\begin{equation*}
  \begin{tikzcd} 
  & ((1+1)+(1+1)) \ar[mapsto]{dl}[swap]{\mu} \ar[mapsto]{dr}{TTe} \\
   (1+1) + (1+1) \ar[mapsto]{d}[swap]{\mu} \ar[mapsto, near end, swap]{dr}{Te} & & (2+2) \ar[mapsto]{dl}[near end]{\mu} \ar[mapsto]{d}{Te} \\
   1+1+1+1 & 2+2 & 4
  \end{tikzcd}
 \end{equation*}
In order to get the witness of the composite partial evaluation, we have to then remove the inner nesting, or inner brackets, via the map $T\mu$.
The element thus obtained is $(1+1+1+1)$, as expected:
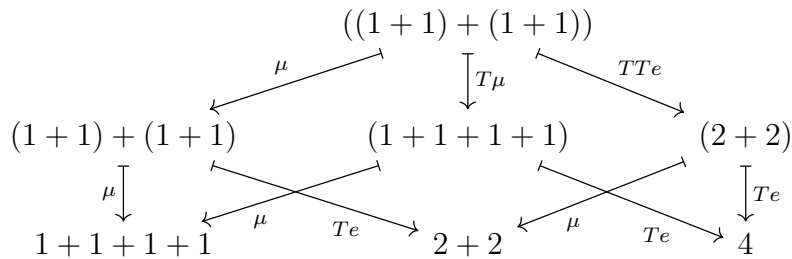
\begin{equation*}
  \begin{tikzcd} 
  & ((1+1)+(1+1)) \ar[mapsto]{dl}[swap]{\mu} \ar[mapsto]{d}{T\mu} \ar[mapsto]{dr}{TTe} \\
   (1+1) + (1+1) \ar[mapsto]{d}[swap]{\mu} \ar[mapsto, near end, swap]{dr}{Te} & (1+1+1+1) \ar[mapsto]{dl}[near end]{\mu} \ar[mapsto, near end, swap]{dr}{Te} & (2+2) \ar[mapsto]{dl}[near end]{\mu} \ar[mapsto]{d}{Te} \\
   1+1+1+1 & 2+2 & 4
  \end{tikzcd}
 \end{equation*}
\end{eg}

Let's now try to generalize the example above, and the role of ``substitution''. First, let's introduce some terminology. We recall the definition of a cartesian monad.

\begin{deph}[{e.g.~\cite[Section~4.1]{leinster}}]
 Let $(T,\eta,\mu)$ be a monad on a category $\cat{C}$. The monad $(T,\eta,\mu)$ is called \emph{cartesian} if:
 \begin{itemize}
  \item The functor $T$ preserves pullbacks;
  \item All naturality squares of $\eta$ and $\mu$ are pullbacks.
 \end{itemize}
\end{deph}

\begin{eg}
	Every monad on $\cat{Set}$ which arises from a (non-symmetric) operad is cartesian~\cite[Section~C.1]{leinster}. 
\end{eg}

As we will see, partial evaluations for cartesian monads are particularly well-behaved. But cartesianness is also a very restrictive condition, and we thus consider a variant of it, based on the standard concept of weak pullback.

\begin{deph}
 Let $\cat{C}$ be a category. A commuting square in $\cat{C}$
 \begin{equation*}
  \begin{tikzcd}
   A \ar{r}{f} \ar{d}{g} & B \ar{d}{m} \\
   C \ar{r}{n} & D
  \end{tikzcd}
 \end{equation*}
 is a \emph{weak pullback} if for every object $S$ of $\cat{C}$ and pair of arrows $b:S\to B$ and $c:S\to  C$ such that $m\circ b=n\circ c$, there exists an arrow $a: S\to A$ such that the following diagram commutes:
 \begin{equation*}
  \begin{tikzcd}
  S \ar[swap]{dddr}{c} \ar[near end]{dr}{a} \ar{drrr}{b} \\ 
   & A \ar{rr}{f} \ar{dd}{g} && B \ar{dd}{m} \\
   \\
   & C \ar{rr}{n} && D
  \end{tikzcd}
 \end{equation*}
\end{deph}

Note that we do not require the map $a$ to be unique. We use this to generalize the notion of weakly cartesian monad~\cite{weakweber,seldomweak} from $\cat{Set}$ to all categories.

\begin{deph}
 Let $(T,\eta,\mu)$ be a monad on a category $\cat{C}$. We say that $(T,\eta,\mu)$ is \emph{weakly cartesian} if:
 \begin{itemize}
  \item The functor $T$ preserves weak pullbacks;
  \item The naturality squares of $\eta$ and $\mu$ are weak pullbacks.
 \end{itemize}
\end{deph}

\begin{eg}
The free monoid monad on $\cat{Set}$ is cartesian~\cite[Observation~2.1(d)]{seldomweak}. The free commutative monoid monad on $\cat{Set}$ is weakly cartesian, but not cartesian~\cite[Example~8.2]{seldomweak}.
\end{eg}

We have the following result:

\begin{prop}
 \label{trans_prop}
 Let $T$ be a monad on a category $\cat{C}$. Let $(A,e)$ be a $T$-algebra, and suppose that the following diagram is a weak pullback:
 \begin{equation}\label{topdiamond}
  \begin{tikzcd}
   TTTA \ar{r}{TTe} \ar{d}{\mu} & TTA \ar{d}{\mu} \\
   TTA \ar{r}{Te} & TA
  \end{tikzcd}
 \end{equation}

Then the partial evaluation relation on every set of formal expressions $\cat{C}(S,TA)$ is transitive.
\end{prop}

\begin{proof}
We have to prove the following: given $p,q,r:S\to TA$ and $k,h:S\to TTA$ such that $(Te)\circ k=p$, $\mu\circ k=(Te)\circ h=q$, and $\mu\circ h=r$. Then there exists $\rho:S\to TTA$ such that $Te\circ\rho=p$ and $\mu\circ\rho=r$. 

 Consider now the commutative diagram:
 \begin{equation}\label{cubeupbar}
  \begin{tikzcd}[column sep=large]
	  S \ar[dotted]{rrr}{q} \ar[near start,dotted,bend right=10em,"r" description]{rrrrddd} \ar[dotted,swap]{rrddd}{p} \ar[dotted,"k" description]{rrd} \ar[dotted,"h" description,near start]{rrrrd} & & & TA \\
	  & & TTA \ar[crossing over]{ur}{\mu} \ar[crossing over,swap]{dd}{Te} && TTA  \ar[swap]{ul}{Te} \ar{dd}{\mu} \\
	  & & & TTTA \ar[swap]{ul}{TTe} \ar[swap]{ur}{\mu} \ar[crossing over]{dd}{T\mu} \\
   & & TA && TA \\
   & & & TTA \ar{ul}{Te} \ar[swap]{ur}{\mu}
  \end{tikzcd}
 \end{equation}
 (which commutes by the composition, associativity, and naturality squares). 
 Then we have that $p$ sits in the bottom left corner, $q$ in the top corner, and $r$ in the bottom right corner, while $k$ sits in the top left corner, and $h$ in the top right. Since the top diamond is exactly diagram~\eqref{topdiamond}, which by hypothesis is a weak pullback diagram, there exists an $a:S\to TTTA$ such that~\eqref{cubeupbar} still commutes. Therefore $\rho:=(T\mu)\circ a$ is such that $(Te)\circ \rho =p$ and $\mu \circ \rho=r$.
\end{proof}

Since the square~\eqref{topdiamond} is necessarily a weak pullback for any weakly cartesian monad, we have

\begin{cor}
Let $T$ be a weakly cartesian monad. Then for every $T$-algebra $(A,e)$ and every object $S$, the partial evaluation relation on $\cat{C}(S,TA)$ is transitive.
\end{cor}

Since the free commutative monoid monad is weakly cartesian, this construction reproduces Example~\ref{egcompose}.

In the same way, if $T$ is cartesian, or only if $\mu$ is, then the diagram~\eqref{topdiamond} is a pullback. This makes the composition of partial evaluations into an \emph{algebraic} operation, if we keep track of the element of $TTA$ which witnesses each partial evaluation relation.

\section{As a simplicial object}

The diagram~\eqref{cubeupbar} essentially encodes the first three levels of the \emph{bar construction}~\cite[Section~VII.6]{maclane}. In particular, the weak pullback condition of Proposition~\ref{trans_prop} is exactly a Kan filler condition, as for nerves of categories and more generally quasicategories, applied to the bar construction. More generally, the bar construction has the flavor of a higher-categorical extension of the partial evaluation relation. The study of its higher compositional properties is work in progress; in this section, we describe what we know so far. 

\begin{deph}
 Let $T$ be a monad on $\cat{C}$ and $(A,e)$ a $T$-algebra. The \emph{bar construction of $(A,e)$} is the simplicial object $A_\bullet$ in the category of $T$-algebras given by the following assignments for all $i\ge 0$:
 \begin{itemize}
  \item $A_i := T^{i+1}A$;
  \item $d_j:A_{i+1}\to A_{i}$ given by\begin{itemize}
                                        \item $T^{j}\mu:T^{i+2}A \to T^{i+1}A$ for $0\le j< i+1$, and
                                        \item $T^{i+1}e:T^{i+2}A \to T^{i+1}A$ for $j= i+1$;
                                       \end{itemize}
  \item $s_j:A_{i}\to A_{i+1}$ given by $T^{j+1}\eta: T^{i+1}A \to T^{i+2}A$ for $0\le j \le i$.
 \end{itemize}
\end{deph}

The simplicial identities are guaranteed to hold by the monad and algebra structure, and by naturality of the structure maps. This implies, in particular, that given an object $S$ of $\cat{C}$, we get a simplicial set $\cat{C}(S,A_\bullet)$, with the following interpretation:
\begin{itemize}
 \item The vertices of the simplicial set are given by the (generalized) elements of $TA$, i.e.~(generalized) formal expressions;
 \item The 1-simplices are given by witnesses of partial evaluations: since the source and target maps
	 \[
		d_0, d_1 \: : \:  \cat{C}(S,A_1) \longrightarrow \cat{C}(S,A_0)
	 \]
	 are exactly given by applying $\mu$ and $Te$ as in the definition of the partial evaluation relation, Definition~\ref{defpev}. We can hence view the 1-simplices of the bar construction as arrows pointing in the direction of partial evaluation;
 \item For every vertex, or equivalently formal expression, the map
	 \[
		 s_0 \: : \: \cat{C}(S,A_0) \longrightarrow \cat{C}(S,A_1)
	 \]
	 given by $T\eta$ gives an ``identity'' 1-simplex, which has the correct source and target, corresponding to the proof of Proposition~\ref{trivialpev}. 
 \item The composition of 1-simplices, when defined as in the proof of Proposition~\ref{trans_prop}, is given by a 2-simplex which is exactly a Kan filler of an inner horn. When $T$ is weakly cartesian, this filler always exists. When $T$ is cartesian, this filler is moreover unique, and the resulting simplicial set is even the nerve of a category~\cite{segal}.
\end{itemize}

Since partial evaluations (or equivalently, partial decompositions) are in most cases intrinsically directed, these simplicial objects (and the simplicial sets that we obtain by proving them with objects $S$) give ``spaces'' which are intrinsically directed. As spaces, it thus seems most natural to study them with the methods and tools of \emph{directed homotopy theory} (see for example~\cite{directedat}).

\section{In terms of rewriting systems}

An \emph{abstract reduction system} (ARS) is a set equipped with a binary relation, called the \emph{reduction relation}~\cite[Chapter~1]{rewrbook}.

Given a monad $T$ on $\cat{C}$, a $T$-algebra $(A,e)$, and an object $S$, partial evaluations of $S$-indexed expressions give a binary relation on $\cat{C}(S,TA)$. The results of Section~\ref{secpev} then give the following.
\begin{prop}
 The set $\cat{C}(S,TA)$ equipped with the partial evaluation relation $\to$ gives an ARS with the following properties:
 \begin{itemize}
  \item Reflexivity: for each $s\in \cat{C}(S,TA)$, $s\to s$;
  \item Confluence: if for $s,t,u \in \cat{C}(S,TA)$ we have
  
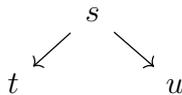
\begin{equation*}
   \begin{tikzcd}[row sep=small, column sep=small]
    & s \ar{dr}\ar{dl} \\
    t && u 
   \end{tikzcd}
  \end{equation*}
  then there exists $z\in \cat{C}(S,TA)$ such that
  
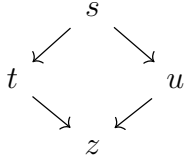
\begin{equation*}
   \begin{tikzcd}[row sep=small, column sep=small]
    & s \ar{dr}\ar{dl} \\
    t \ar{dr} && u \ar{dl} \\
    & z 
   \end{tikzcd}
  \end{equation*}
  In particular, $z$ can always be given by the ``total evaluation'' $\eta\circ e\circ t = \eta\circ e\circ u$ of Proposition~\ref{totalev};
  \item If $T$ is weakly cartesian, then $\to$ is transitive. 
 \end{itemize}
\end{prop}

The composition of partial evaluations can be thought of as a ``rewriting of rewritings'', which points to the theory of \emph{higher rewritings} (see for example \cite{burroni_polygraphs,3drewriting}). At least in this framework, however, higher rewrite rules are defined in terms of simplicial rather than globular shapes.

\section{Examples}

Here are some examples of monads and of the partial evaluations that they induce. Here we restrict to monads on $\cat{Set}$, and traditional elements (as in, arrows from the terminal set $1\to X$ as opposed to more general arrows $S\to X$).

\paragraph{Monoid and group action monads.}
Let $G$ be a monoid (or group) in $\cat{Set}$. The same example works more generally for internal monoids (or groups) in a cartesian monoidal category in terms of generalized elements, but we explain it only in the case of $\cat{Set}$ for simplicity.

The assignment $X\mapsto G\times X$ is part of a functor equipped with a monad structure, with unit and multiplication induced by those of $G$, and the algebras $e:G\times A\to A$ are the objects equipped with $G$-actions. 
 Let now $(g,x)$ and $ (h,y)$ be elements of $G\times A$. We have that $(h,y)$ is a partial evaluation of $(g,x)$ if and only if there is an element $(h,\ell,x)\in G\times G\times A$ such that $h\ell=g$ and $\ell x=y$. In pictures:
 \begin{equation*}
  \begin{tikzcd}[,row sep=huge]
   x \ar[bend left]{rr}{g} \ar{r}[swap]{\ell} & \ell x \ar{r}[swap]{h} & g x
  \end{tikzcd}
 \end{equation*}
 In other words, $(h,y)$ is a partial evaluation of $(g,x)$ if and only if we can write $g$ as a composite $h\ell$, such that ``applying only the part $\ell$ to $x$ gives $y$''. So $(h,y)$ is ``further along in the orbit'' than $(g,x)$. 
 If $G$ is a group, then whenever $x$ and $y$ are on the same orbit we can find the decomposition above, by setting $\ell=h^{-1}g$, and the partial evaluation relation is symmetric: it is the equivalence relation given by belonging to the same orbit. If $G$ is only a monoid, instead, then the partial evaluation relation is generally stronger than being in the same orbit and need not be symmetric. 

As this monad (on $\cat{Set}$) is associated to a non-symmetric operad, it is a cartesian monad. Thus witnesses of partial evaluations can be uniquely composed (and their composition is given by the composition of the monoid or group). In this case, the category whose nerve is the simplicial set $\cat{Set}(1,A_\bullet)$ arising from the bar construction has pairs $(g,x)$ as above as objects, and triples $(h,\ell,x)$ as above as morphisms, with domain $(h\ell,x)$ and codomain $(h,\ell x)$.

\paragraph{Idempotent monads.} For idempotent monads, it is easy to check that all partial evaluations are trivial, in the sense that the partial evaluation relation is the equality relation.

\paragraph{Free monoid monad.} This monad is also associated to an operad. Therefore, also here, partial evaluations can be uniquely composed, and form a category. We are currently not aware of a more explicit description of this category.

\paragraph{Free commutative monoid monad.} This monad is weakly cartesian~\cite{seldomweak}. Therefore, partial evaluation witnesses in $TTA$ can still be composed, but the composition is typically not unique.

\section{Partial evaluations in probability}

Partial evaluations for probability monads permit to compare probability distributions in terms of \emph{how spread} or \emph{how random} they are. 

Common ways of measuring the ``randomness'' of a probability measures are functionals like variance and entropy. However, there is important information that a single real number cannot encode. Intuitively, a single number can measure only ``how much'' the randomness is, but not ``where'', or ``in which way''. 

\begin{eg}
 Consider for example the probability distributions on $\R$ whose densities are represented in the following picture.
 
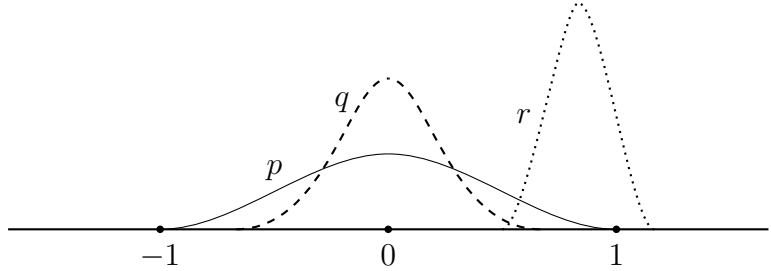
\begin{center}
 \begin{tikzpicture}[baseline=(current  bounding  box.center),>=stealth]
  \draw[-,thick] (-5,0) -- (5,0);
  \node[bullet,label=below:$-1$] (0) at (-3,0) {};
  \node[bullet,label=below:$0$] (0) at (0,0) {};
  \node[bullet,label=below:$1$] (0) at (3,0) {};
  
  \draw[-] (-3,0) .. controls (-2,0) and (-1,1) .. (0,1) node[midway,above] {$p$};
  \draw[-] (3,0) .. controls (2,0) and (1,1) .. (0,1) ;
  
  \draw[-,thick,dashed] (-2,0) .. controls (-0.8,0) and (-0.5,2) .. (0,2) node[near end,left] {$q$};
  \draw[-,thick,dashed] (2,0) .. controls (0.8,0) and (0.5,2) .. (0,2) ;
  
  \draw[-,thick,dotted] (1.5,0) .. controls (1.8,0) and (2.3,3) .. (2.5,3) node[midway,left] {$r$};
  \draw[-,thick,dotted] (3.5,0) .. controls (3.2,0) and (2.8,3) .. (2.5,3) ;
 \end{tikzpicture}
 \end{center}
	One can say that $p$ is ``more random'' or ``more spread'' than $q$ \emph{around the same center of mass}. Instead, while $r$ looks more ``peaked'' than $q$, it is so ``somewhere else'': it has indeed less randomness \emph{quantitatively}, but over \emph{different regions}. In a partial order, we would say that $q$ and $r$ are incomparable. In higher dimensions, the same would be true if the two distributions were spread around the same center of mass, but along different directions. This is what we mean by ``where the randomness is''. 
\end{eg}

It turns out that partial evaluations can be employed to detect this finer notion of randomness, and that they are related to the so-called \emph{second-order stochastic dominance} relation~\cite{fishburn}.
In the rest of this section we will sketch how this works. The details have been worked out in the second author's PhD thesis \cite[Chapter~4]{thesis}.
Again we will focus on set-theoretic elements rather then generalized elements.

\subsection{The idea of probability monads}

The idea of using monads in the context of probability theory, as we describe it here, goes back to Lawvere~\cite{early} and Giry~\cite{giry}.

We have seen that monads can be interpreted in terms of formal expressions encoding possible ``operations''. In the case of probability monads, the operations in question are formal \emph{convex combinations}, or \emph{mixtures}.

Consider a coin flip, where ``heads'' and ``tails'' both have probability $1/2$. Then \emph{in some sense}, this is a convex combination of ``heads'' and ``tails''. Formally, the set $\{\mbox{``heads''},\mbox{``tails''}\}$ is not set in which convex combinations are defined, so one can't really take actual mixtures of its elements. However, one can embed $\{\mbox{heads},\mbox{tails}\}$ into the space
\begin{equation*}
 \big\{ \lambda\,\mbox{``heads''} + (1-\lambda)\,\mbox{``tails''} \; | \: \lambda \in [0,1] \big\},
\end{equation*}
using the map which sends
\begin{align*}
	\mbox{``heads''} & \mapsto 1\,\mbox{``heads''} + 0\,\mbox{``tails''}, \\
	\mbox{``tails''} & \mapsto 0\,\mbox{``heads''} + 1\,\mbox{``tails''}.
\end{align*}
In \emph{this} new space, one can actually take convex combinations: for example, $1/2\,\mbox{``heads''} + 1/2\,\mbox{``tails''}$ is now a convex combination of the extremal points $\mbox{``heads''}$ and $\mbox{``tails''}$.
In general one does not only take finite convex combinations, but rather integrals with respect to normalized measures, so we are talking about \emph{generalized} mixtures, in the sense of Choquet theory~\cite{winkler}. 
The interpretation is nevertheless the same.
\begin{itemize}
 \item Given an object $X$, which we can think of a set of possible (deterministic) states, we can form an object $PX$, which contains ``formal mixtures'' of elements of $X$;
 \item Every function $f:X\to Y$ gives a function $Pf:PX\to PY$ by convex-linear extension;
 \item $X$ is embedded into $PX$ via a map $\delta:X\to PX$ which maps an element $x\in X$ to the trivial formal convex combination $x$;
 \item Formal mixtures of formal mixtures can be evaluated using the map $E:PPX\to PX$, as the following example illustrates.
\end{itemize}

\begin{eg}\label{egcoins}
 Suppose that you have two coins in your pocket. Suppose that one coin is fair, with ``heads'' on one face and ``tails'' on the other; suppose the second coin has ``heads'' on both sides. Suppose now that you draw a coin randomly, and flip it. 
 
 We can sketch the probabilities in the following way:
 
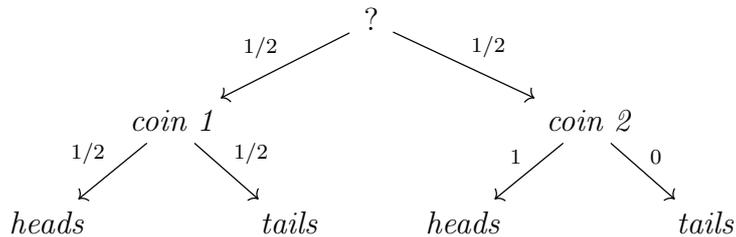
\begin{equation*}
  \begin{tikzcd}[column sep=tiny]
   &&& ? \ar{dll}[swap]{1/2} \ar{drr}{1/2} \\
   & \mbox{coin 1} \ar{dl}[swap]{1/2} \ar{dr}{1/2} &&&& \mbox{coin 2} \ar{dl}[swap]{1} \ar{dr}{0} \\
   \mbox{heads} && \mbox{tails} && \mbox{heads} && \mbox{tails}
  \end{tikzcd}
 \end{equation*}
 Let $X$ be the set $\{\mbox{``heads''},\mbox{``tails''}\}$. A coin gives a \emph{law} according to which we will obtain ``heads'' or ``tails'', so it determines an element of $PX$. Since the choice of coin is also random (we also have a \emph{law on the coins}), the law on the coins determines an element of $PPX$.
 
 By averaging, the resulting overall probabilities are 
 
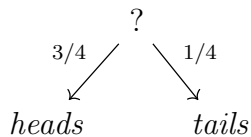
\begin{equation*}
  \begin{tikzcd}[column sep=tiny]
   & ? \ar{dl}[swap]{3/4} \ar{dr}{1/4} \\
   \mbox{heads} && \mbox{tails} 
  \end{tikzcd}
 \end{equation*}
 In other words, the ``average'' or ``composition'' can be thought of as an assignment $E:PPX\to PX$, from laws of ``random random variables'' to laws of ordinary random variables.
\end{eg}

There are spaces, for example $\R$, where one \emph{can} take actual mixtures. These correspond exactly to the \emph{algebras} of $P$. In other word, a $P$-algebra is a space with (finite and possibly suitably infinite) convex combinations satisfying suitable equations, such as a convex subset of some vector space. Taking expectation values is one of the most important operations in probability theory: the spaces where this can be done are precisely the algebras of a probability monad. 

The details of how this is carried out in practice vary, depending on the choice of category, of monad, and so on. So in particular, one may get different sorts of ``convex spaces''. The probability monad that we use in this section, the Kantorovich monad, has as algebras precisely the closed convex subsets of Banach spaces (see~\cite{ours_kantorovich}). Another example in the literature is the Radon monad on the category of compact Hausdorff spaces: its algebras are precisely the compact convex subsets of locally convex topological vector spaces~\cite{swirszcz,keimel}. 

\subsection{Two examples}

We recall here the definition of two probability monads: the distribution monad on $\cat{Set}$, of use in theoretical computer science, and the Kantorovich monad, which one can consider as an analogue with possibly continuous distributions, in the category of complete metric spaces. 

First we sketch the basic construction of the \emph{distribution monad}, also known as the \emph{convex combination monad} or \emph{finitary Giry monad}.

Let $X$ be a set. Define $D X$ as the set whose elements are functions $p:X\to[0,1]$ such that $p(x)\ne 0$ for only finitely many $x$, and $\sum_{x\in X} p(x)=1$.
Note that the sum above is finite if one excludes all the vanishing terms.
The elements of $D X$ are called \emph{finite distributions} or \emph{finitely supported probability measures} over $X$. 

Given a function $f:X\to Y$, one defines the \emph{pushforward} $D f:D X\to D Y$ as follows. Given $p\in D X$, then $(D f)(p)\in D Y$ is the function
$$
y \;\mapsto\; \sum_{x\in f^{-1}(y)} p(x) .
$$
This makes $D$ into an endofunctor on $\cat{Set}$. 
The unit map $\delta:X\to D X$ maps the element $x\in X$ to the function $\delta_x:X\to[0,1]$ given by
$$
\delta_x(y) \;=\; \begin{cases}
1 & y=x ;\\
0 & y\ne x .
\end{cases}
$$
The multiplication map $E:D D X\to D X$ maps $\xi\in D D X$ to the function $E\xi\in D X$ given by
$$
E\xi(x) \;=\; \sum_{p\in D X} p(x) \, \xi(p).
$$
The maps $E$ and $\delta$ satisfy the usual monad laws.
The $D$-algebras are known as \emph{convex spaces}, and their morphisms as \emph{affine} or \emph{convex-linear maps}. 
For more details, see for example~\cite{fritz} and~\cite{jacobs}.

The Kantorovich monad is a monad on the category $\cat{CMet}$ of complete metric spaces and short (nonexpanding) maps, i.e.~functions $f:X\to Y$ such that for every $x,x'\in X$,
$$
d\big( f(x),f(x') \big) \le d(x,x') .
$$

\begin{deph}
 Let $X$ be a complete metric space. The \emph{Kantorovich-Wasserstein space} $PX$ is the space whose elements are Radon probability measures on $X$ with finite first moment, and whose metric is given by:
 $$
 d(p,q) := \sup_{f:X\to\R} \int f \, dp - \int f \, dq ,
 $$
 where the supremum is taken over all the short maps $X\to \R$.
\end{deph}

The assignment $X\mapsto PX$ is part of a functor: we can assign to each morphism $f:X\to Y$ a morphism $Pf:PX\to PY$ given by the push-forward of probability measures. In other words, if $p\in PX$ and $A$ is a measurable subset of $Y$, then:
$$
(Pf)(p)(A) := (f_*p)(A) = p(f^{-1}(A)) .
$$

The unit of the monad is given by the Dirac delta map $\delta:X\to PX$, which assigns to each $x\in X$ the Dirac mass $\delta_x$ concentrated at $X$. The composition $E:PPX\to PX$ is given by integration, as in Example~\ref{egcoins}: if $\mu\in PPX$ and $A$ is a measurable subset of $X$, then
$$
(E\mu)(A) := \int_{PX} p(A) \, d\mu(p) .
$$

The algebras of the Kantorovich monad must be first of all objects of our category, i.e.\ complete metric spaces. Moreover, as we have seen, they should be closed with respect to convex combinations in some sense, as for example convex regions of a vector space. Closed convex subsets of Banach spaces are then an ideal candidate: they are complete metric spaces, and they are convex. It can be proven~\cite[Section~5.3]{ours_kantorovich} that the $P$-algebras in $\cat{CMet}$ are \emph{exactly} the closed convex subsets of Banach spaces (up to isomorphism), with the structure map given by the (Bochner) integral. 

For more details, we refer the reader to~\cite{ours_kantorovich} and~\cite{thesis}.

\subsection{Partial expectations, dilations, conditional expectations}

Let's now study partial evaluations for algebras of the two probability monads mentioned above. 
The material in this subsection about the Kantorovich monad be found more in detail in~\cite[Section~4.2.1]{thesis}, in particular, all the proofs can be found there (modulo some differences in terminology). The result is closely related to previous work of Winkler and Weizsäcker (see~\cite{winkler} and the discussion therein). 

First of all, for both monads the partial evaluation relation is transitive.

\begin{prop}
Consider a function $f:X\to Y$. The following commutative diagram is a weak pullback.
$$
\begin{tikzcd}
 DDX \ar{d}{\mu} \ar{r}{DDf} & DDY \ar{d}{\mu} \\
 DX \ar{r}{Df} & DY
\end{tikzcd}
$$
\end{prop}

\begin{proof}
Let $p\in DX$, let $\alpha\in DDY$ be given as the formal convex combination
$$
\alpha \coloneqq \sum_{q\in DY} \alpha_q\,\delta \left( \sum_{y\in Y} q_y\, \delta_y \right). 
$$
Let $p=\sum_x p_x \, (x)$, and suppose that $\mu(\alpha)=f_*p$, which means that for every $y\in Y$,
$$
 \sum_{q\in DY} \alpha_q \, q_y = \sum_{x\in f^{-1}(y)} p_x . 
$$

Take now the measure $\beta\in DDX$ given by
$$
\beta \coloneqq \sum_{q\in DY} \alpha_q \,\delta \left( \sum_{y\in Y} q_y \cdot p|_y \right). 
$$

We have
$$
f_{**}\beta = \sum_{q\in DY} \alpha_q\, \delta \left( \sum_{y\in Y} q_y \,f_*( p|_y) \right) = \sum_{q\in DY} \alpha_q\,\delta \left( \sum_{y\in Y} q_y \,\delta_y \right) = \alpha.
$$

Just as well,
\begin{align*}
 \mu(\beta) &= \sum_{q\in DY} \alpha_q  \sum_{y\in Y} q_y \cdot p|_y  = \sum_{y\in Y} \sum_{q\in DY} \alpha_q \, q_y \cdot p|_y \\
 &= \sum_{y\in Y} \sum_{x\in f^{-1}(y)} p_x \cdot p|_y = \sum_{y\in Y} (f_*p)_y \, \delta_y \cdot p|_y = f_*p \cdot p|_y = p . 
\end{align*}
\end{proof}

The analogous statement for the Kantorovich monad is the following result~\cite[Theorem~2.6.9]{thesis}:

\begin{thm}
 For every naturality square
\[\begin{tikzcd}
	TTX \ar{r}{E} \ar{d}{TTf} & TX \ar{d}{Tf} \\
	TTY \ar{r}{E} & TY
\end{tikzcd}\]
 of the multiplication transformation $E:PP\Rightarrow P$, the weak universality property required of a weak pullback holds for maps out of the singleton space $1$.
\end{thm}

This is enough to see that the partial evaluation relation for algebras of the Kantorovich monad is transitive at the level of ordinary elements. But more generally, we do not know:

\begin{prob}
Is the multiplication of the Kantorovich monad weakly cartesian?
\end{prob}

Given a set or complete metric space $X$, let $p,q$ be distributions in $DX$ or $PX$, respectively. The intuition behind a partial evaluation from $p$ to $q$ is that $q$ is ``more concentrated'' than $p$, or ``closer to a delta at its center of mass''. From the statistical point of view, $q$ is better approximated by just looking at its expectation than $p$, since $p$ is ``more spread out''. 

Just as well, also the inverse process, partial decomposition, is useful in probability. It is known and it goes under the name of a \emph{dilation}: a random map which intuitively ``only spreads, but does not translate'' (think of diffusion without drift, or the kernel of a martingale). In statistics, this roughly corresponds to ``adding unbiased noise'', or ``casual, not systematic errors''. 

\begin{deph}
 In the category of sets, let $(A,e)$ be a $D$-algebra (for example, $A=\R$), and let $p\in DA$. A \emph{$p$-dilation} is a map $k:A\to DA$ such that for all $a$ in the support of $p$, $e(k(a))=a$. 
\end{deph}

\begin{deph}
 In the category of complete metric spaces, let $(A,e)$ be a $P$-algebra. For $p\in PA$, a \emph{$p$-dilation} is a map $k:A\to PA$, $e(k(a))=a$ for $p$-almost all $a \in A$. 
\end{deph}

Trivially, every dilation is a $p$-dilation. The most trivial dilations are the unit components of $D$ and $P$. 

\begin{prop}
 In the category of sets, let $(A,e)$ be a $D$-algebra, and let $p,q\in DA$. The following conditions are equivalent:
 \begin{enumerate}
  \item\label{pevcond} There exists a partial decomposition of $p$ into $q$;
  \item\label{dilcond} There exists a $p$-dilation $k$ such that $E\circ Dk (p)=q$.
 \end{enumerate}
\end{prop}
\begin{proof}
 \begin{itemize}
  \item $\ref{pevcond}\Rightarrow\ref{dilcond}$: Let $r\in DDA$ be a partial decomposition of $p$ into $q$, i.e.~$De(r)=p$ and $E(r)=q$. Construct a map $k:A\to DA$ as follows.
  For each $s\in DA$ which lies in the support of $r$, note that $e(s)\in A$ lies in the support of $p=De(r)$. 
  Now for each $a$ in the support of $p$, define $k(a)\in DA$ by ``conditioning'', i.e.
  $$
  k(a)(b) \coloneqq \dfrac{1}{p(a)} \sum_{s\in e^{-1}(a)} r(s) \,s(b) 
  $$
  for all $b\in A$.
  The value of $k$ outside the support of $p$ can be arbitrary, and we get a $p$-dilation. Moreover, 
  $$
  Dk(p)(s) = \sum_a p(a)\, k(a) = r(s) ,
  $$
  so that $Dk(p)=r$, and so $E\circ Dk(p)=q$.
  \item $\ref{dilcond}\Rightarrow\ref{pevcond}$: Define $r\coloneqq Dk(p)$. We have that $E(r)=q$ by hypothesis. Moreover, for each $a$ in the support of $p$, $e(k(a))=a$, so that 
  $$
  De(Dk(p)) (a) = \sum_a p(a)\,e(k(a)) = p(a),
  $$
  hence $De(r)=p$. Therefore, $r$ is a partial decomposition of $q$ into $p$.
 \end{itemize}
\end{proof}

We have then again the following result for the continuous case \cite[Lemma~4.2.17]{thesis}.

\begin{thm}
 In the category of complete metric spaces, let $(A,e)$ be a $P$-algebra, and let $p,q\in PA$. The following conditions are equivalent:
 \begin{enumerate}
  \item There exists a partial decomposition of $p$ into $q$;
  \item There exists a $p$-dilation $k$ such that $E\circ k_* p=q$.
 \end{enumerate}
\end{thm}

We have gained an extra interpretation: in the context of probability, a partial decomposition of $p$ into $q$ is a process of ``adding noise'', or ``letting diffusion take place''. Conversely, we can then also interpret partial evaluations as ``removing noise''.

In probability theory there exists already a concept that intuitively is a ``partial expectation'', namely, \emph{conditional expectation} of random variables. It turns out that the two concepts are in some sense equivalent, at least in the case of the Kantorovich monad.

\begin{deph}\label{defce}
 Consider a probability space $(X,\mathcal{F},\mu)$, a sub-$\sigma$-algebra $\mathcal{G}$ of $\mathcal{F}$, and measurable mappings $f,g:X\to A$ such that $f_*\mu$ and $g_*\mu$ have finite first moment.
 We say that $g$ is a \emph{conditional expectation of $f$ given $\mathcal{G}$} if:
 \begin{itemize}
  \item The function $g$ is also $\mathcal{G}$-measurable;
  \item For every $G$ in the $\sigma$-algebra $\mathcal{G}$, we have
  \begin{equation*}
   \int_G g \, d\mu = \int_G f \, d\mu .
  \end{equation*}
 \end{itemize}
\end{deph}

For brevity, we extend the terminology to the image measures themselves:
\begin{deph}
 Let $p,q\in PA$. We call a \emph{conditional expectation} of $p$ into $q$ \emph{in distribution} a probability space $(X,\mathcal{F},\mu)$ together with a sub-$\sigma$-algebra $\mathcal{G}$ of $\mathcal{F}$, and mappings $f,g:X\to A$, with $f$ $\mathcal{F}$-measurable and $g$ $\mathcal{G}$-measurable, such that $p=f_*\mu$, $q=g_*\mu$, and $g$ is a conditional expectation of $f$ given $\mathcal{G}$. 
\end{deph}

Here is now the main result~\cite[Theorem~4.2.14]{thesis}:

\begin{thm}\label{eqce}
 Let $(A,e)$ be a $P$-algebra, and let $p,q\in PA$. The following conditions are equivalent:
 \begin{enumerate}
  \item There exists a partial evaluation of $p$ into $q$;
  \item There exists a conditional expectation of $p$ into $q$ in distribution.
 \end{enumerate}
\end{thm}

So, in particular, the law of total evaluation of Proposition~\ref{totalev} corresponds to the well-known law of total expectation of random variables.

This does not mean, however, that whenever there is a partial evaluation of $p$ into $q$, their associated \emph{random variables} are in relationship of conditional expectation: we are only looking at the distributions, and we are not even taking them to be jointly distributed: while it can be shown that both a partial evaluation of $p$ into $q$ and a conditional expectation in distribution do determine a joint distribution, the mere \emph{existence} of either kind of structure does not. In other words, the theorem does not give an equivalence of \emph{structures} (partial evaluations and conditional expectations), but merely an equivalence of \emph{properties} of admitting those structures. The question of whether the equivalence can be strengthened to an equivalence of structures, up to suitable isomorphism, is currently still open.

Yet again an additional interpretation can be given, connecting to the notion of \emph{second-order stochastic dominance}. 
We can interpret a concave function as a ``risk-averse observer'', as it is customary in mathematical finance~\cite{risk,fishburn}. The reason is that the integral of a convex function is higher if the integration measure is more ``concentrated''.  

\begin{deph}
 Let $(A,e)$ be a $P$-algebra, and let $p,q\in PA$. We say that $p\le q$ in the \emph{second-order stochastic dominance relation} if and only if for every concave function $f:A\to\R$, 
 $$
 \int f\, dp \;\le\; \int f\, dq .
 $$
\end{deph}

We have that this order is again equivalent to the partial evaluation order \cite[Theorem~4.4.9]{thesis}:

\begin{thm}
 Let $(A,e)$ be a $P$-algebra, and let $p,q\in PA$. We have that $p\le q$ in second-order stochastic dominance if and only if there is a partial evaluation from $p$ to $q$.
\end{thm}

Therefore, the partial evaluation order for probability distributions on Banach spaces encodes in a categorical way exactly also the notion of randomness that is used in mathematical finance, native to that field.

\section*{Acknowledgements}
\addcontentsline{toc}{section}{Acknowledgements}

This paper was originally written for the Applied Category Theory 2019 school. We thank the participants of our and other groups for their interest, useful feedback and fruitful collaboration. We also thank Dirk Hofmann, Joachim Kock, Steve Lack, Rostislav Matveev, Paige Randall North, Sharwin Rezagholi, David Spivak, and Tarmo Uustalu for the very fruitful discussions and the helpful advice. Most of this paper was written while both authors were with the Max Planck Institute for Mathematics in the Sciences, which we thank for providing an outstanding research environment.

\bibliographystyle{alpha}
\bibliography{catprob}
\addcontentsline{toc}{section}{\refname}

\end{document}